\documentclass[11pt]{amsart} 
\usepackage[left=3.2cm,right=3.2cm,bottom=3.2cm,top=3.2cm]{geometry}

\usepackage{amsmath,amsfonts,amsthm,amssymb,mathrsfs,enumerate,comment, amscd}
\usepackage{cite}

\theoremstyle{plain}
\newtheorem{theorem}{Theorem}

\newtheorem*{lemma*}{Lemma}
\newtheorem{lemma}[theorem]{Lemma}

\newtheorem*{theorem*}{Theorem}

\newtheorem{proposition}[theorem]{Proposition}
\newtheorem*{proposition*}{Proposition}

\newtheorem*{corollary*}{Corollary}

\newtheorem*{conja}{Conjecture A}
\newtheorem*{conjb}{Conjecture B}

\theoremstyle{definition}

\newtheorem{remark}[theorem]{Remark}
\newtheorem*{remark*}{Remark}

\newtheorem*{definition*}{Definition}
\newtheorem{definition}[theorem]{Definition}

\newtheorem*{example*}{Example}

\numberwithin{equation}{section}

\def\RR{{\mathbb R}}

\def\codim{\operatorname{codim}}

\def\im{\operatorname{im}}

\def\c1{\operatorname{c_1}}
\def\c2{\operatorname{c_2}}

\def\Sym{\operatorname{Sym}}

\def\CC{{\mathbb C}}

\def\QQ{{\mathbb Q}}
\def\PP{{\mathbb P}}

\def\O{{\mathscr O}}

\def\E{{\mathscr E}}

\def\F{{\mathscr F}}

\def\Supp{\operatorname{Supp}}
\def\Alb{\operatorname{Alb}}

\def\+{\oplus}                   
\def\*{\otimes}                  

\def\eff{\operatorname{Eff}}
\def\nef{\operatorname{Nef}}

\def\Pic{\operatorname{Pic}}

\def\Sym{\operatorname{Sym}}

\def\Supp{\operatorname{Supp}}

\def\ME{\operatorname{ME}}
\def\NE{\operatorname{NE}}
\def\Jac{\operatorname{Jac}}
\def\mob{\operatorname{mob}}

\usepackage{mathtools}

\DeclarePairedDelimiter{\floor}{\lfloor}{\rfloor}

\title{On subvarieties with ample normal bundle}
\author{John Christian Ottem}
\address{Department of Pure Mathematics and Mathematical Statistics, University of Cambridge, Wilberforce Road, Cambridge CB3 0WB, UK}
\email{J.C.Ottem@dpmms.cam.ac.uk}
\begin{document}

\begin{abstract}
We show that a pseudoeffective $\RR$-divisor has numerical dimension 0 if it is numerically trivial on a subvariety with ample normal bundle. This implies that the cycle class of a curve with ample normal bundle is big, which gives an affirmative answer to a conjecture of Peternell. We also give other positivity properties of such subvarieties.
\end{abstract}

\maketitle
\thispagestyle{empty}

A well-established principle in algebraic geometry is that geometric properties of an algebraic variety is reflected in the subvarieties which are in various senses `positively embedded' in it. The primary example is the hyperplane section in a projective embedding of the variety, which gives rise to the notion of an ample divisor. However, in higher codimension it is less clear what it should mean in general for a subvariety to be `ample'. In his book \cite{Har70}, Hartshorne considers several approaches, including the condition that the normal bundle of the subvariety should be an ample vector bundle. Even for divisors this condition is weaker than ampleness, in the sense that it is a condition that concerns a vector bundle on the subvariety itself, rather than a global condition on the ambient variety. Still the condition guarantees certain good properties of the subvariety, e.g., that its cycle class is nef.  For a more global definition of ampleness, see \cite{Ott12}.

Our first result is the following theorem, which gives a positive answer to a question of Peternell \cite{Pet11}. The theorem essentially says that a pseudoeffective divisor which is numerically trivial on a subvariety with ample normal bundle is very far from being big.

\begin{theorem}\label{interiortheorem}
Let $X$ be a smooth projective variety over an algebraically closed field of characteristic 0 and let $Y$ be a smooth subvariety of dimension $>0$ with ample normal bundle. If $D$ is a pseudoeffective $\RR$-divisor such that $D|_Y\equiv 0$, then its numerical dimension $\nu(D)$ is 0. In particular, if $D$ is nef, then $D\equiv 0$.
\end{theorem}
See Definition \ref{numericaldimension} for the precise definition of numerical dimension of a divisor. In particular, this implies that the Iitaka dimension $\kappa(D)$ is non-positive.

Combining this result with the duality theorem of  Boucksom--Demailly--Paun--Peternell \cite{BDPP}, we prove the following result about curves with ample normal bundle. The first part of the theorem was  also conjectured by Peternell \cite{OP04,Pet08,Pet11}.

\begin{theorem}\label{interiortheorem2}
Let $X$ be a smooth projective variety over $\CC$, let $C$ be a smooth curve with ample normal bundle. Then the cycle class of $C$ is big, i.e., it lies in the interior of the cone of curves, $\overline{\NE}(X)$.

If in addition $C$ is strictly nef (i.e., $C\cdot D>0$ for any irreducible divisor $D$), then the cycle class of $C$ lies in the interior of the cone of movable curves, $\overline{\ME}(X)$. In particular, $C$ is numerically equivalent to a $\QQ$-linear combination of strongly movable curves.
\end{theorem}

%
%


Interestingly, Voisin  \cite{Voi08} showed that the corresponding result is false for subvarieties of higher dimensions. More precisely, she gives an examples of smooth projective varieties in any dimension $\ge 4$, containing a codimension 2 subvariety with ample normal bundle, but whose class is in the boundary of the pseudoeffective cone.  In these examples,  the subvariety deforms in a family covering the ambient variety and the normal bundle is even globally generated.


The strictly nef assumption in the second part of the theorem is necessary. Indeed, take any smooth projective variety with a curve with ample normal bundle and blow up a point outside it. Then on the blow-up, the preimage of the curve, $C$, has ample normal bundle, but the exceptional divisor satisfies $E\cdot C=0$, so $C$ lies in the boundary of the cone of movable curves.

%

On the other hand, the following theorem says that there can be at most finitely many prime divisors disjoint from $C$.

\begin{theorem}\label{finitelymany}Let $X$ be a smooth projective variety over $\CC$ and let $Y\subset X$ be  a smooth subvariety of dimension at least one with ample normal bundle. Then $Y$ intersects all but finitely many prime divisors on $X$. In fact, the number of such divisors is less than the Picard number of $X$.
\end{theorem}



All of the above results remain valid if $Y$ is assumed to be locally complete intersection instead of smooth.

Thanks to Fr\'ed\'eric Campana and Burt Totaro for comments and useful discussions.

%
%



%
\section{Curves with positivity properties}

Let $X$ be a smooth projective variety. An $\RR$-divisor is a finite sum $D=\sum \lambda_i D_i$ where $\lambda_i\in \RR$ and each $D_i$ is an irreducible divisor in $X$. Write $N^1(X)=\Pic(X)\otimes \RR/\equiv$ for the N\'eron-Severi group of $X$, i.e., the $\RR$-vector space of divisors modulo numerical equivalence. In $N^1(X)$ we define the effective cone $\eff(X)$ as the cone spanned by effective divisors and the nef cone $\nef(X)$ the cone of nef divisors, i.e., $\RR$-divisors such that $D\cdot C\ge 0$ for every curve $C$ on $X$. An $\RR$-divisor is \emph{pseudoeffective} if its class lies in the closure $\overline{\eff}(X)$ of the effective cone.

We let $N_1(X)$ denote the vector space of 1-cycles modulo numerical equivalence, and $\NE(X)$ the cone spanned by curves on $X$. We call a cycle $\alpha\in N_1(X)$ \emph{big} if it lies in the interior of $\NE(X)$. Inside $\NE(X)$, there is the subcone $\ME(X)$ spanned by curves that are \emph{movable}. Here a curve $C$ is called movable if it is a member of a family of curves that dominates $X$.  By definition,  the cones $\nef(X)$ and $\overline{\NE}(X)$ are dual with respect to the intersection pairing. A fundamental result of Boucksom--Demailly--Paun--Peternell\cite{BDPP}, states that for a smooth variety over $\CC$, also the cones $\overline{\ME}(X)$ and $\overline{\eff}(X)$ are dual, i.e., a divisor $D$ is pseudoeffective if and only if $D\cdot C\ge 0$ for all movable curves $C$. Moreover, they show that $\overline{\ME}(X)$ coincides with the closure of the cone spanned by curves which are strongly movable, that is, 1-cycles of the form $f_*(H_1\cap \cdots \cap H_{n-1})$, where the $H_i$ are very ample divisors on $X'$ and $f:X'\to X$ is birational.

\subsection{Subvarieties with ample normal bundle}
Recall that a vector bundle $\E$ on a variety is \emph{ample} if the line bundle $\O(1)$ is ample on $\PP(\E)$. Here and throughout the paper we use the Grothendieck notation for projectivized bundles, i.e., $\PP(\E)$ is the variety of hyperplanes in $\E$. If $\E$ is a vector bundle on a curve, $\E$ is ample if and only if every quotient line bundle of $\E$ has positive degree \cite{Harcurves}. We will mainly consider the case when $\E$ is the normal bundle $N_Y=(\mathcal I/\mathcal I^2)^*$, of a subvariety $Y\subset X$, which is a vector bundle of rank equal to the codimension when $Y$ is smooth (or more generally locally complete intersection.)

Subvarieties with ample normal bundle share many interesting geometric properties  with ample divisors (see e.g., \cite{Har70} or \cite{Laz04}). For example, for every coherent sheaf $\F$, the cohomology groups $H^i(X-Y,\F)$ are finite-dimensional vector spaces for $i<\dim Y$ \cite{Har70}.  Also, if $\dim Y\ge 1$, a result of Napier and Ramachandran \cite{NR98} says $\im(\pi_1(Y)\to \pi_1(X))$ has finite index in $\pi_1(X)$. A property which will be important for our purposes is the following:

\begin{proposition}\cite[Corollary 8.4.3]{Laz04}
Let $Y\subset X$ be a subvariety with ample normal bundle, then $Y$ is nef, i.e., $Y\cdot Z\ge 0$ for any subvariety with $\dim Y+\dim Z=\dim X$.
\end{proposition}
In his book \cite{Har70}, Hartshorne presented two of his influential conjectures about such subvarieties:

\begin{conja}[Hartshorne]\label{Hartshorne2}
 Let $Y\subset X$ be a smooth subvariety of $X$ such that the normal bundle $N_{Y}$ is 
an ample vector bundle. Is it true that some multiple of $Y$ deforms (as a cycle) in a family covering 
$X$? 
\end{conja}

\begin{conjb}[Hartshorne]\label{Hartshorne1}
Let $Y,Z$ be smooth subvarieties of $X$ with such that the normal bundles $N_Y,N_Z$ are ample vector bundles. If $\dim Y+\dim Z\ge \dim X$, then $Y\cap Z\neq \emptyset$.
\end{conjb}

It is known by results of \cite{FL81} that Conjecture A implies Conjecture B. Unfortunately, Fulton and Lazarsfeld also showed that Conjecture A is false in general when $\dim Y\ge 2$. Their counterexample is based on constructing a certain ample rank two vector 
bundle on $Y=\PP^2$, so that no multiple of the zero-section moves in the total space of the bundle. 

Given this, one asks whether it is still true that a \emph{curve} $C$ with ample normal bundle 
has a multiple that moves in $X$. This question is open in general, but is known when $X$ is a surface or when $g(C)\le 1$. Further evidence for this is given by the result of Campana and Flenner \cite{CF90}, which states that some multiple of the zero-section moves in the normal bundle. In particular, this implies that the example of Fulton and Lazarsfeld cannot be modified to the dimension 1 case.

Theorem \ref{interiortheorem2} can be viewed as a weak form of Hartshorne's Conjecture A. Indeed, since the cycle class of $C\subset X$ is big, a multiple of it can be written as $h+e$ where $h$ is the class of a complete intersection of $n-1$ sufficiently ample divisors and $e$ is an effective 1-cycle. In particular, for any finite set of points in $X$, there is an effective cycle numerically equivalent to $mC$ which passes through them.  Note also that if $C$ is in addition assumed to be strictly nef, Theorem \ref{interiortheorem2} implies that some integral multiple of $C$ is even equivalent to a sum of strongly movable curves in $X$.

%
%

\subsection{Examples}

(i) If $D$ is a divisor with ample normal bundle, then $D$ is nef and big \cite{Voi08}. In particular, when $X$ is a surface, Theorems \ref{interiortheorem} and \ref{finitelymany} follow directly from the Hodge index theorem. 

(ii) If $Y$ is a complete intersection, or more generally, a transverse intersection of subvarieties with ample normal bundle, then $N_Y$ is ample. 

 (iii) Any smooth subvariety of projective space has ample normal bundle, since it is the quotient of the tangent bundle $T_{\PP^n}$ which is ample \cite{Laz04}. 


(iv) If the normal bundle $N_C$ is sufficiently ample in the sense that $h^0(N_C)\ge \dim X-1$ and $h^1(N_C)=0$, then the curve $C$ itself moves in a family covering $X$. In this case it is known that $C$ is big by \cite[Theorem 4.11]{Pet11}. In particular, this holds when $C$ is rational or elliptic. In fact, a variety is rationally connected if and only if it contains a rational curve with ample normal bundle. In \cite{OP04}, Oguiso and Peternell give an analogous geometric characterization when $C$ is an elliptic curve in a threefold.

 (v) If $C$ has genus $g\ge 2$, we can consider the embedding of $C$ in its Jacobian $\Jac(C)$. Here the normal bundle of $C$ is ample \cite{Laz04}.  In this example, it is classically known that the cycle-class of $C$ is in the interior of the cone of curves of $\Jac(C)$. In fact, Poincare's formula gives that $C\equiv \frac1{(g-1)!}\Theta^{g-1},$ where $\Theta$ is the theta divisor of $\Jac(C)$, which is ample.
 
(vi) If $X$ is a homogenous manifold, then the ampleness of the normal bundle of a subvariety can often be interpreted geometrically. For example, $Y$ is non-degenerate. If $X$ is an abelian variety and $C$ is a curve, then $N_C$ is ample if and only if a translate of $C$ generates $X$ as a group \cite{Laz04}. If $X$ is a quadric, then by \cite[Theorem 1]{Ballico}, the normal bundle $N_{C}$ is ample if and only if $C$ is not a line. In general, a line in a homogeneous manifold has ample bundle if and only if $X=\PP^n$.

(vii) Bigness of the cycle class of $C$ has however no implications for the positivity of the normal bundle. Indeed, take any 3-fold with Picard number one containing a $(-1,-1)$ curve: then $N_C=\O(-1)\oplus \O(-1)$, and $C$ is big because $\overline{\NE}(X)$ is 1-dimensional.

\section{Proof of Theorem \ref{interiortheorem} and \ref{interiortheorem2}}

\subsection{Divisorial Zariski decomposition}

We briefly recall the divisorial Zariski decomposition introduced by Boucksom \cite{Bou04} and Nakayama \cite{Nak04}.  Let $X$ be a smooth projective variety and let $D$ be a pseudoeffective $\RR$-divisor. We define the \emph{diminished base locus} of $D$ by $$\mathbf B_-(D)=\bigcup_{A} \mathbf B_\RR(D+A),$$ where $A$ runs over all ample divisors and $\mathbf B_\RR(D)=\bigcap\{\Supp(D') | D'\ge 0\mbox{ and } D'\sim_\RR D\}$. By \cite[Theorem V.1.3]{Nak04}, $\mathbf B_-(D)$ is a countable union of closed subsets. Let $H$ be an ample line bundle on $X$. For each prime divisor $\Gamma$ on $X$ define the coefficient
$$\sigma_\Gamma(D)=\lim_{\epsilon\to 0^+}\inf\{\mbox{mult}_\Gamma(D') | D'\sim_\RR D+\epsilon H \mbox{ and }D'\ge 0\}
$$It was shown by Nakayama \cite[III.1.5]{Nak04} that these numbers do not depend on the choice of $H$ and that there are only finitely many prime divisors $\Gamma$ such that $\sigma_{\Gamma}(D)>0$. Following \cite{Nak04} we then define $N_\sigma(D)=\sum_\Gamma \sigma_\Gamma(D)\Gamma$ and $P_\sigma(D)=D-N_\sigma(D)$, and call $D=N_\sigma(D)+P_\sigma(D)$ the \emph{divisorial Zariski decomposition} of $D$. 

The main properties of this decomposition is captured by the following 

\begin{proposition}\cite[III.1.4, III.1.9, V.1.3]{Nak04} Let $D$ be a pseudoeffective $\RR$-divisor.
\begin{enumerate}[(i)]
\item $N_\sigma(D)$ is effective and $\Supp(N_
\sigma(D))$ coincides with the divisorial part of $\mathbf B_-(D)$.
\item $N_\sigma(D)=0$ when $D$ is nef.
\item For all $m\ge 0$, $H^0(X,\O_X(\floor{mP_\sigma(D)}))\simeq H^0(X,\O_X(\floor{mD}))$
\end{enumerate}
\end{proposition}

\begin{definition}\label{numericaldimension}
 Let $D$ be a pseudoeffective $\RR$-divisor. For an ample divisor $H$ define $\nu(D,H)$ as the maximal non-negative integer $k$ such that 
 $$
\limsup_{m\to\infty}\frac{h^0(X,\O_X(\floor{mD}+H)}{m^k}>0
 $$We define the \emph{numerical dimension} $\nu(D)$ has the maximal value of $\nu(D,H)$ when $H$ varies over all ample divisors on $X$. (Although the paper \cite{BDPP} uses a different definition of $\nu(D)$, it is equivalent to ours by the main theorem in \cite{Leh11}.)
\end{definition}

%

\begin{lemma}\label{slowgrowth}\cite[Proposition V.2.7]{Nak04}
Let $X$ be a smooth projective variety and let $D$ be a pseudoeffective $\RR$-divisor. Then $\nu(D)=0$ if and only if $D\equiv N_\sigma(D)$.
\end{lemma}

Since this result is vital in the proof of Theorem 1, we give a proof in the case $D$ is a nef divisor. In fact, this special case is enough to prove the first part of Theorem 2. We will prove the following statement: If $H$ is a smooth very ample divisor, then $D\equiv 0$ if and only if for all sufficiently large $k$, $\nu(D,kH)=0$. We'll use the observation that $D\equiv 0$ if and only if $D|_H\equiv 0$ (which comes from the fact that $D\equiv 0$ if and only if $D\cdot H^{n-1}=0$). By Fujiita's vanishing theorem, there is a $k_0$ such that $H^1(X,\O_X(mD+(k-1)H))=0$ for all $m\ge0$ and $k\ge k_0$. Consider now the restriction map
$$H^0(X,\O_X(mD+kH))\to H^0(H,\O_H(mD+kH)).$$By construction, this map is surjective for every $m\ge 0,k\ge k_0$, so in particular also $\nu(D|_H,kH|_H)=0$ for all $k\ge k_0$. By induction on the dimension, $D|_H\equiv 0$ and hence also $D\equiv 0$. 

When $D$ is only pseudoeffective, essentially the same idea can be used, but a different vanishing theorem is required (cf. \cite{Nak04}).

\begin{lemma}\label{h0vanishing}
 Let $\E$ be an ample vector bundle on a curve $C$ and let $d$ be an integer. Then there is an integer 
 $m_0=m_0(d)>0$ so that 
 $$
 H^0(C,\Sym^m \E^* \otimes  \O_C(L))=0
 $$for all $m\ge m_0$, and all line bundles $L$ of degree $d$.
\end{lemma}

%
%

%
\begin{proof}
Let $\PP(\E)$ denote the variety of hyperplanes in $\E$ with projection $\pi:Y\to C$. By the ampleness of $\E$, the line bundle $\O_{\PP(\E)}(1)$ is ample on $\PP(\E)$. Hence by Serre duality and the Leray spectral sequence,
\begin{eqnarray*}
 H^0(C,\Sym^m \E^* \otimes  \O_C(L))&=&H^1(C,\O_C(K_C-L)\otimes \Sym^m(\E))\\
 &=&H^1({\PP(\E)},\pi^*(K_C-L)\otimes \O(m))=0\end{eqnarray*}
 The last cohomology group vanishes for all $m\ge m_0$, where $m_0$ depends only on $d$ (e.g., by Fujita's vanishing theorem \cite{Laz04}).
\end{proof}

Note that proof uses the characteristic $0$ assumption  in the isomorphism $(\Sym^m \E^*)^*=\Sym \E$.



\begin{lemma}\label{h0bounded}
Let $C\subset X$ be a smooth curve with ample normal bundle and let $D$ be a pseudoeffective $\RR$-divisor on $X$ such that $D\cdot C=0$. Then for any ample divisor $H$, the function $h(t)= h^0(X,\O_X(\floor{tD}+H))$ is bounded.
\end{lemma}

\begin{proof}Let $I$ be the ideal sheaf of $C$ in $X$. Since $C$ is locally complete intersection, we have $I^{k}/I^{k+1}=\Sym^k N_C^*$. By taking global sections of the exact sequences
$$
0\to I^{k+1} (\floor{tD}+H)\to I^{k}(\floor{tD}+H) \to \Sym^k N_C^*\otimes \O_C(\floor{tD}+H)\to 0
$$for $k=0,1,\ldots$, we deduce that 
\begin{equation*}\label{h0sum}
h^0(X,\O_X(\floor{tD}+H))\le \sum_{k=0}^\infty h^0(C, \Sym^k N_C^* \otimes  \O_C(\floor{tD}+H))
\end{equation*}Note that we have $\floor{tD}\cdot C\le tD\cdot C=0$. So in particular, $\deg \O_C(\floor{tD}+H)$ is bounded above by some constant $K>0$ depending only on $D$ and $H$. 

By Lemma \ref{h0vanishing}, there is a $k_0\ge 1$ so that the cohomology groups on the right-hand side of \eqref{h0sum} vanish for $k\ge k_0$ and all $t$. In particular, 
\begin{equation*}\label{h0sum2}
h^0(X,\O_X(\floor{tD}+H))\le \sum_{k=0}^{k_0} h^0(C, \Sym^k N_C^* \otimes  \O_C(\floor{tD}+H))
\end{equation*}Moreover, as each of the terms on the right-and side are bounded above by a constant independent of $t$, we see that the same holds for $h^0(X,\O_X(\floor{tD}+H))$.
\end{proof}


With these results, we are now in position to prove Theorem 1 and 2.

\begin{proof}[Proof of Theorem \ref{interiortheorem}]It suffices to prove the theorem when $Y$ is a curve. Indeed, if $\dim Y\ge 2$ and  $A_1,\dots,A_{\dim Y-1}$ are sufficiently general, smooth, ample divisors, then $C=Y\cap A_1\cap \cdots \cap A_{\dim Y-1}$ will be a smooth curve and $D|_Y\equiv 0$ if and only if $D\cdot C=0$. Moreover, the normal bundle of $C$ is ample, because it is an extension of the ample vector bundles $N_{Y|X}|_C$ and $N_{C|Y}$ (see e.g., \cite[III.\S 1]{Har70}).

So suppose that $Y=C$ is a curve with ample normal bundle and let $D$ be a pseudoeffective $\RR$-divisor such that $D\cdot C=0$ and let $H$ be any ample divisor. By Lemma \ref{h0bounded}, we have that the dimensions of the cohomology groups $H^0(X,\O_X(\floor{tD}+H))$ are bounded above, so in particular $\nu(D)=0$. Moreover, if $D$ is nef, from the definition, $N_\sigma(D)=0$, so in particular $D\equiv 0$.\end{proof}

\begin{proof}[Proof of Theorem \ref{interiortheorem2}]
Let $C$ be a curve with ample normal bundle. By definition, the cone of curves $\overline{\NE}(X)\subset N_1(X)$ is dual to $\nef(X)\subset N^1(X)$. Hence, to show that the class of $C$ is in the interior of the cone of curves it suffices to show that if $D$ is a nef $\RR$-divisor such that $D\cdot C=0$, then $D\equiv 0$. But this is exactly the first part of Theorem \ref{interiortheorem}.

Suppose now that $C$ is strictly nef (i.e., $C\cdot D>0$ for all effective divisors $D$), we need to show that the class of $C$ is in the interior of the cone of \emph{movable  curves}, $\overline{\ME}(X)\subset N_1(X)$. By \cite{BDPP}, the movable cone is dual to the pseudoeffective cone, so we need only check that $C\cdot D>0$ for every pseudoeffective $\RR$-divisor which is not numerically trivial. Let $D$ be a pseudoeffective $\RR$-divisor such that $C\cdot D=0$. By Lemma \ref{slowgrowth} and Lemma \ref{h0bounded}, we have that $D\equiv N_\sigma(D)=\sum \sigma_\Gamma \Gamma$, so in particular also $N_\sigma(D)\cdot C=0$, contradicting the strictly nefness of $C$.\end{proof}
\def\mob{\operatorname{mob}}

\begin{remark} The paper \cite{Ott12} presents a definition of ampleness for subschemes of arbitrary codimension, generalizing the usual notion for divisors.  In short, a subscheme is defined to be \emph{ample} if  the exceptional divisor on the blow-up along the subscheme satisfies a certain partial positivity condition, namely that its asymptotic cohomology groups vanish in certain degrees (it is `$q$-ample' in the sense of \cite{Tot10}, with $q=\codim Y-1$). When $Y$ is smooth, or locally complete intersection, it is known that this condition implies that the normal bundle of $Y$ is ample and $Y$ is strictly nef \cite[Corollary 5.6]{Ott12}.

\end{remark}

\section{Proof of Theorem \ref{finitelymany}}Let $X$ be a smooth complex variety over $\CC$ and let $Y$ be a smooth subvariety with ample normal bundle and let $D\subset X$ be any effective divisor (reducible or non-reduced) such that $Y\cap D=\emptyset$. By Theorem \ref{interiortheorem}, $D$ must have numerical dimension 0, so in particular its Iitaka dimension $\kappa(D)$ is also 0. From this we have

\begin{lemma}\label{uniquelin}
Let $D$ be an effective divisor disjoint from $Y$. Then $H^0(X,\O_X(D))=\CC$, i.e., $D$ is the unique effective divisor in its linear equivalence class.
\end{lemma}

%

The following lemma is the essential ingredient in the proof of Theorem \ref{finitelymany}. The idea of using the Albanese variety was inspired by an argument used by Totaro \cite{Tot00}.

\begin{lemma}
Let $Y\subset X$ be a smooth subvariety with ample normal bundle. Then the restriction map \begin{equation}\label{restrH1}
 H^1(X,\QQ)\to H^1(Y,\QQ)
\end{equation}
 is injective.\end{lemma}

\begin{proof}This essentially follows since a subvariety with ample normal bundle can not be contracted to a point by a non-constant morphism. Fix a base-point on $Y$ and consider the map of Albanese varieties
\begin{equation*}
\alpha: \Alb(Y)\to \Alb(X).
\end{equation*}If  \eqref{restrH1} is not injective, then $\alpha$ is not surjective, i.e., the quotient abelian variety $B=\Alb(X)/\alpha(\Alb(Y))$ has positive dimension. Note that the composition$$Y\to X\to \Alb(X)\to B$$sends $Y$ to a point $b\in B$. Let $f$ be a non-constant holomorphic function in a neighbourhood of $b$, which vanishes at $b$. Note that the above composition pulls the function $f$ back to a global section of $H^0(Y,I_Y^m/I_Y^{m+1})$ for some $m>0$. But $I_Y^m/I_Y^{m+1}=\Sym^m N_C^*$ cannot have global sections if the  normal bundle of $Y$ is ample.
\end{proof}

In particular, this implies that the map of abelian varieties 
\begin{equation}\label{finiteker}
 \Pic^0(X)\to \Pic^0(Y)
\end{equation}has finite kernel.

\begin{lemma}\label{uniquenum}
Suppose $D_1,D_2$ are numerically equivalent effective divisors whose supports are disjoint from $Y$. Then $D_1$ and $D_2$ are equal as divisors.\end{lemma}
\begin{proof}Suppose first that $D_1$ and $D_2$ are algebraically equivalent. By definition, the element $D_1-D_2$ defines an element of $\Pic^0(X)$. Note that $D_1-D_2$ restricts to $0$ in $\Pic^0(Y)$ (since both $D_1$ or $D_2$ are disjoint from $Y$). Since the kernel of \eqref{finiteker} is finite, this means that there is a positive integer $m>0$ such that $m(D_1-D_2)=0$ in $\Pic^0(X)$ and hence $mD_1$ and $mD_2$ are linearly equivalent. By Lemma \ref{uniquelin}, we have $mD_1=mD_2$, and also $D_1=D_2$.

If $D_1$ and $D_2$ are numerically equivalent, then by Matsusaka's theorem, there is an integer $m>0$ such that $mD_1$ and $mD_2$ are algebraically equivalent. Using the same argument again, we find that $D_1=D_2$.
\end{proof}

%
%
%
%
%
%
%

With this we can complete the proof of Theorem \ref{finitelymany}:
\begin{proof}[Proof of Theorem \ref{finitelymany}]After replacing $Y$ with an appropriate linear section in some projective embedding, we may assume that $Y$ is a smooth curve. We may also suppose that the Picard number $\rho$ is greater than 1, otherwise there is nothing to prove. 

Now take any distinct $\rho$ prime divisors $D_1,\ldots,D_\rho$ disjoint from $Y$. Since $D_i\cdot Y=0$ for $i=0,\ldots, \rho$, we see that the $D_i$ lie in a  rational hyperplane in $N^1(X)$. Hence after re-ordering the $D_i$, there is a relation of the form
$$
m_1D_1+\cdots+m_s D_s \equiv m_{s+1} D_{s+1}+\cdots+m_\rho D_\rho
$$where $m_i$ are non-negative integers. Now let $E$ (resp. $F$) denote the divisor on the left hand side (resp. right hand side) of this equation. Note that the supports of $E$ and $F$ are disjoint from $Y$, so by Lemma  \ref{uniquenum} the divisors $E,F$ are equal. This contradicts the assumption that the components $D_1,\ldots,D_\rho$ are different.
 \end{proof}
 


%

%
%
%
%
%

\bibliographystyle{alpha}

\end{document}